\documentclass[11pt]{amsart}
\usepackage{a4wide,amssymb,color}
\usepackage[all]{xy}

\newcommand{\E}{\mathcal E}
\newcommand{\F}{\mathcal F}
\newcommand{\U}{\mathcal U}
\newcommand{\V}{\mathcal V}
\newcommand{\W}{\mathcal W}
\newcommand{\I}{\mathcal I}
\newcommand{\A}{\mathcal A}
\newcommand{\M}{\mathcal M}
\newcommand{\IR}{\mathbb R}
\newcommand{\IC}{\mathbb C}
\newcommand{\IE}{\mathbb E}
\newcommand{\so}{\mathsf{so}}
\newcommand{\supp}{\mathrm{supp}}
\newcommand{\Homeo}{\mathrm{Homeo}}

\newcommand{\Asigma}{\mathsf \Sigma}
\newcommand{\Adelta}{\mathsf \Delta}
\newcommand{\Alambda}{\mathsf \Lambda}
\newcommand{\Ddelta}{\widehat\Adelta}
\newcommand{\Dsigma}{\widehat\Asigma}

\newcommand{\e}{\varepsilon}
\newcommand{\w}{\omega}
\newcommand{\Ra}{\Rightarrow}

\newcommand{\add}{\mathrm{add}}
\newcommand{\non}{\mathrm{non}}
\newcommand{\cov}{\mathrm{cov}}
\newcommand{\cof}{\mathrm{cof}}

\newtheorem{theorem}{Theorem}[section]
\newtheorem{problem}[theorem]{Problem}

\newtheorem{lemma}[theorem]{Lemma}

\theoremstyle{definition}
\newtheorem{remark}[theorem]{Remark}
\newtheorem{example}[theorem]{Example}

\title{Set-Theoretical Problems in Asymptology}
\author{Taras Banakh and Igor Protasov}
\dedicatory{Dedicated to the memory of Kenneth Kunen}

\address{T.Banakh: Ivan Franko National University of Lviv (Ukraine) and Jan Kochanowski University in Kielce (Poland)}
\email{t.o.banakh@gmail.com}
\address{I.Protasov: Faculty of Computer Science and Cybernetics, Taras Shevchenko National University of  Kyiv, Academic Glushkov pr. 4d, 03680 Kyiv, Ukraine}
\email{i.v.protasov@gmail.com}

\begin{document}
\begin{abstract}
In this paper we collect some open set-theoretic problems that appear in the large-scale topology (called also Asymptology). In particular, we ask problems on: (i) critical cardinalities of some special (large, indiscrete, inseparated) coarse structures on $\omega$, (ii) the interplay between properties of a coarse space and its Higson corona, (iii) some special ultrafilters ($T$-points and cellular $T$-points) related to finitary coarse structures on $\omega$, (iv) partitions of coarse spaces into thin pieces, and (v) coarse groups having various extremal properties.
\end{abstract} 
\maketitle

\section*{Introduction}
In this paper we collect selected open problems in the large-scale topology that have set-theoretic flavor. The large-scale topology (or else Asymptology \cite{PZ}) studies properties of coarse spaces. Coarse spaces defined in term of balls were introduced under name balleans in \cite{PB} independently and simultaneously with \cite{Roe}. The necessary preliminary information on coarse spaces is collected in Section~\ref{s:coarse}. More information related to large-scale topology can be found in the monographs \cite{PB}, \cite{PZ}, \cite{Roe}.
In Section~\ref{s:critical} we pose some open problems on the critical cardinalities of some special (large, inseparated, indiscrete) coarse structures on $\w$ and in Section~\ref{s:Higson} we investigate the interplay between properties of a coarse space and its Higson corona. Section~\ref{s:ultra} is devoted to some special ultrafilters on $\w$, and Section~\ref{s:d} to the interplay between coarse properties of an action of a group $G\subseteq S_\omega$ on $\w$ and the induced action of $G$ on the compact Hausdorff space $\omega^*=\beta\w\setminus\w$. In Section~\ref{s6} we study partitions of coarse spaces into finitely many thin pieces and in the final Section~\ref{s:groups} we collect some questions about coarse groups with various  extremal properties.


\section{Coarse structures and coarse spaces}\label{s:coarse}

Coarse structures can be introduced via balls as in \cite{PB}, \cite{PZ} or via entourages as in \cite{Roe}.

An {\em entourage} on a set $X$ is any subset $E$ of the square $X\times X$ that contains the diagonal $\Delta_X:=\{(x,x):x\in X\}$ of $X\times X$. For two entourages $E,F$, the sets 
$$E^{-1}=\{(y,x):(x,y)\in E\}\mbox{ \ and \ }EF=\{(x,z):\exists y\in X\;\;(x,y)\in E\;\wedge (y,z)\in F\}$$are entourages on $X$.

For an entourage $E$ on $X$, a point $x\in X$ and a set $A\subseteq X$, the sets
$$E(x):=\{y\in X:(x,y)\in E\}\mbox{ \ and \ }E[A]=\bigcup_{a\in A}E(a)$$are called the $E$-balls around $x$ and $A$, respectively.

An entourage $E$ is called 
\begin{itemize}
\item {\em symmetric} if $E=E^{-1}$;
\item {\em cellular} if $E^{-1}=E=EE$ (i.e., $E$ is an equivalence relation on $X$);
\item {\em locally finite} if for any $x\in X$ the set  $E^\pm(x):=E(x)\cup E^{-1}(x)$ is finite;
\item {\em finitary} if the cardinal $\sup_{x\in X}|E^\pm(x)|$ is finite;
\item {\em trivial} if the set $E\setminus\Delta_X$ is finite.
\end{itemize}
It is clear that each cellular entourage is symmetric, each trivial entourage is finitary, and each finitary entourage  is locally finite.

A {\em coarse structure} on a set $X$ is any family $\E$ of entourages on $X$  satisfying the following axioms:
\begin{itemize}
\item[(A1)] for any $E,F\in\E$, the entourage $EF^{-1}$ is contained in some entourage $G\in\E$;
\item[(A2)] $\bigcup\E=X\times X$;
\item[(A3)] for any entourages $E\subseteq F$ on $X$ the inclusion $F\in\E$ implies $E\in\E$.
\end{itemize}
A subfamily $\mathcal B\subseteq\E$ is called a {\em base} of the coarse structure $\E$ if each entourage $E\in\E$ is contained in some entourage $B\in\mathcal B$. Any base of a coarse structure satisfies the axioms (A1), (A2). On the other hand, any family of entourages $\mathcal B$ satisfying the axioms (A1), (A2) is a base of the unique coarse structure
$${\downarrow}\mathcal B=\bigcup_{B\in\mathcal B}\{E\subseteq X\times X:\Delta_X\subseteq E\subseteq B\}.$$
For a coarse structure $\E$ on a set $X$, its {\em weight} $w(\E)$ is defined as the smallest cardinality of a base of $\E$.

A coarse structure $\E$ on a set $X$ is called
\begin{itemize}
\item {\em bounded} if $X\times X\in\E$;
\item {\em locally finite} if each entourage $E\in\E$ is locally finite;
\item  {\em finitary} if each entourage $E\in\E$ is finitary;
\item {\em cellular} if $\E$ has a base consisting of cellular entourages.
\end{itemize}

Each set $X$ carries the {\em smallest coarse structure} consisting of all trivial entourages on $X$. The smallest coarse structure is finitary and cellular. The family $\IE_{<\w}(X)$ (resp. $\IE_{\w}(X)$) of all finitary (resp. locally finite) entourages on $X$ is the largest finitary (resp. locally finite) coarse structure on $X$.

Each action $\alpha:G\times X\to X$ of a group $G$ on a set $X$ induces the {\em finitary coarse structure} $\E_G$ on $X$, which is generated by the base consisting of the entourages
$$(L\times L)\cup\{(x,y)\in X\times X:y\in Fx\}$$where $L$ is a finite subset of $X$ and $F$ is a finite subset of the group $G$ that contain the identity $1_G$ of $G$. The action of the permutation group $S_X$ on $X$ induces the {\em largest finitary coarse structure} $\E_{S_X}=\IE_{<\w}(X)$ on a set $X$. The trivial subgroup $\{1_G\}$ induces the smallest finitary coarse structure on $X$.

The following fundamental fact was proved by Protasov in \cite{b1} (see also \cite{b4} and \cite{b2}).

\begin{theorem}\label{t:G} Every finitary coarse structure on a set $X$ coincides with the finitary coarse structure $\E_G$ induced by the action of some subgroup $G\subseteq S_X$.
\end{theorem}

Every metric space $(X,d)$ carries the canonical coarse structure $\E_d$, which is generated by the base 
$$\big\{\{(x,y)\in X\times X:d(x,y)<\e\}:0<\e<\infty\big\}.$$
A coarse structure $\E$ on a set $X$ is called {\em metrizable} if $\E=\E_d$ for some metric $d$ on $X$.

The following metrizability theorem can be found in \cite[2.1.1]{PZ}.

\begin{theorem} A coarse structure on a set is metrizable if and only if it has a countable base.
\end{theorem}

A {\em coarse space} is a pair $(X,\E)$ consisting of a set $X$ and a coarse structure $\E$ on $X$. A coarse space is called {\em metrizable} (resp. {\em bounded}, {\em cellular}, {\em locally finite}, {\em finitary}) if so is its coarse structure $\E$.

Let $(X,\E)$ be a coarse space. A subset $B\subseteq X$ is called {\em bounded} if $B\subseteq E(x)$ for some $E\in\E$ and $x\in X$. The family of all bounded subsets in a coarse space is called the {\em  bornology} of the coarse space.

For two coarse spaces $(X,\E_X)$ and $(Y,\E_X)$, a function $f:X\to Y$ is called 
\begin{itemize}
\item {\em macro-uniform} if for any entourage $E_X\in\E_X$ there exists an antourage $E_Y\in\E_Y$ such that   $\{(f(x),f(x')):(x,x')\in E_X\}\subseteq E_Y$;
\item an {\em asymorphism} if $f$ is bijective and both maps $f$ and $f^{-1}$ are macro-uniform;
\item a {\em coarse equivalence} if $f$ is macro-uniform and there exists a macro-uniform map $g:Y\to X$  such that $\{(g\circ f(x),x):x\in X\}\subseteq E_X$ and  $\{(f\circ g(y),y):y\in Y\}\subseteq E_Y$ for some entourages $E_X\in\E_X$, $E_Y\in\E_Y$.
\end{itemize}

\begin{theorem}[\cite{BZ}] For any countable locally finite groups $G,H$, the finitary coarse spaces $(G,\E_G)$ and $(H,\E_G)$ are coarsely equivalent.
\end{theorem}





\section{Critical cardinalities related to some special coarse structures on $\w$}\label{s:critical}

In this section we consider some uncountable cardinals that appear as smallest weights of  coarse structures on $\w$ that possess some (pathological) properties.
First we recall the necessary information on cardinal characteristics of the continuum (see \cite{Blass}, \cite{vD}, \cite{Vau} for more details).

For a set $X$, its cardinality is denoted by $|X|$. Cardinals are identified with the smallest ordinals of a given cardinality. We denote by $\w$ and $\w_1$ the smallest infinite and the smallest uncountable cardinals, respectively. By $\mathfrak c$ we denote the cardinality of the real line.  For a set $X$, let $[X]^{<\w}$ be the family of all finite subsets of $X$, and $[X]^\w$  be the family of all countably infinite subsets of $X$. 

For two sets $A,B$, we write $A\subseteq^* B$ if $A\setminus B$ is finite. For two functions $f,g:\w\to\w$ we write $f\le g$ if $f(n)\le g(n)$ for all $n\in\w$, and $f\le^* g$ if the set $\{n\in\w:f(n)\not\le g(n)\}$ is finite.

\smallskip

Now we recall the definitions of some small uncountable cardinals, namely:
\begin{itemize}
\item[$\mathfrak b$]$=\min\{|B|:B\subseteq \w^\w\;\wedge\;\forall f\in\w^\w\;\exists g\in B\;\;g\not\le^* f\}$;
\item[$\mathfrak d$]$=\min\{|D|:D\subseteq \w^\w\;\wedge\;\forall f\in\w^\w\;\exists g\in D\;\;f\le g\}$;
\item[$\mathfrak r$]$=\min\{|R|:R\subseteq[\w]^\w\;\wedge\;\forall a\in[\w]^\w\;\exists r\in R\; \;(r\subseteq^* a\;\vee\;r\subseteq^*\w\setminus a)\}$;
\item[$\mathfrak s$]$=\min\{|S|:S\subseteq[\w]^\w\;\wedge\;\forall a\in[\w]^\w\;\exists s\in S\;\;(a\not\subseteq^*s\;\wedge\;a\not\subseteq^*\w\setminus s)\}$;
\item[$\mathfrak t$]$=\min\{|T|:T\subseteq [\w]^\w$ and $(\forall s,t\in T\;\;s\subseteq^* t$ or $t\subseteq^* s)$ and ($\forall s\in[\w]^\w\;\exists t\in T\;\;s\not\subseteq^* t)\}$;
\item[$\mathfrak u$] is the smallest cardinality of a base of a free ultrafilter on $\w$;
\end{itemize}
\smallskip

The order relations between these cardinals are described by the following  diagram (see \cite{Blass}, \cite{vD}, \cite{Vau}), in which for two cardinals $\kappa,\lambda$ the arrow $\kappa\to\lambda$ indicates that $\kappa\le\lambda$ in ZFC.
\smallskip
$$
\xymatrix{
&\mathfrak s\ar[r]&\mathfrak d\ar[rr]&&\mathfrak c\\
\w_1\ar[r]&\mathfrak t\ar[r]\ar[u]&\mathfrak b\ar[u]\ar[r]&\mathfrak r\ar[r]&\mathfrak u\ar[u]
}
$$

Each family of sets $\I$ with $\bigcup\I\notin\I$ has four basic cardinal characteristics:
\begin{itemize}
\item[] $\add(\I)=\min\{|\A|:\A\subseteq\I\;\wedge\;\bigcup\A\notin\I\}$;
\item[] $\cov(\I)=\min\{|\A|:\A\subseteq\I\;\wedge\;\bigcup\A=\bigcup\I\}$;
\item[] $\non(\I)=\min\{|A|:A\subseteq\bigcup\I\;\wedge\;A\notin\I\}$;
\item[] $\cof(\I)=\min\{|\mathcal J|:\mathcal J\subseteq\I\;\wedge\;\forall I\in\I\;\exists J\in\mathcal J\;(I\subseteq J)\}$.
\end{itemize}
These cardinal characteristics are usually considered for the $\sigma$-ideals $\mathcal M$ and $\mathcal N$ of meager sets and Lebesgue null sets on the real line, respectively.
The cardinal characteristics of the $\sigma$-ideals $\mathcal M$ and $\mathcal N$ are described by the famous Cicho\'n diagram:
$$
\xymatrix{
&\cov(\mathcal N)\ar[r]&\non(\M)\ar[r]&\cof(\M)\ar[r]&\cof(\mathcal N)\ar[r]&\mathfrak c\\
&&\mathfrak b\ar[r]\ar[u]&\mathfrak d\ar[u]\\
\w_1\ar[r]&\add(\mathcal N)\ar[r]\ar[uu]&\add(\M)\ar[r]\ar[u]&\cov(\M)\ar[r]\ar[u]&\non(\mathcal N).\ar[uu]\\
}
$$

Now let us return to coarse spaces and introduce some properties of coarse structures that will lead us to some new cardinal characteristics of the continuum. 

Let $\E$ be a coarse structure on a set $X$. A subset $A\subseteq X$ is called
\begin{itemize}
\item {\em $\E$-bounded} if $A\subseteq E(x)$ for some $E\in\E$ and $x\in X$;
\item {\em $\E$-unbounded} if $A$ is not $\E$-bounded;
\item {\em $\E$-large} if $E[A]=X$ for some $E\in\E$;
\item {\em $\E$-discrete} (or else {\em $\E$-thin}) if for every $E\in\E$ there exists an $\E$-bounded set $B\subseteq A$ such that $A\cap E(a)=\{a\}$ for every $a\in A\setminus B$.
\end{itemize}
Two subsets $A,B\subseteq X$ are called
\begin{itemize}
\item {\em $\E$-close} if there exists an entourage $E\in\E$ such that $A\subseteq E[B]$ and $B\subseteq E[A]$;
\item {\em $\E$-separated} if for every entourage $E\in\E$ the intersection $E[A]\cap E[B]$ is $\E$-bounded.
\end{itemize}

A coarse structure $\E$ on a set $X$ is called
\begin{itemize}
\item {\em indiscrete} if each $\E$-discrete set in $X$ is $\E$-bounded;
\item {\em inseparated} if any  $\E$-unbounded sets $A,B\subseteq X$ are not $\E$-separated;
\item {\em large} if any $\E$-unbounded set in $X$ is $\E$-large.
\end{itemize}
A coarse space is called {\em indiscrete} (resp. {\em inseparated}, {\em large}\/) if so is its coarse structure. 

For any infinite locally finite coarse space we have the implications:
$$\mbox{large $\Ra$ inseparated $\Ra$ indiscrete $\Ra$ nonmetrizable}.$$
The largest finitary coarse structure on $\w$ is large and hence inseparated and indiscrete.
Using indiscrete, inseparated or large coarse structures we can characterize certain known cardinal characteristics of the continuum.

\begin{theorem}[\cite{Ban}] The cardinal
\begin{enumerate} 
\item $\mathfrak b$ is equal to the smallest weight of an indiscrete locally finite coarse structure on $\w$;
\item $\mathfrak b$ is equal to the smallest weight of an inseparated locally finite coarse structure on $\w$;
\item $\mathfrak d$ is equal to the smallest weight of a large locally finite coarse structure on $\w$;
\item $\mathfrak c$ is equal to  the smallest weight of a large finitary coarse structure on $\w$.
\end{enumerate}
\end{theorem}

The smallest weights of indiscrete (or inseparated) finitary coarse structures on $\w$ determine new cardinal characteristics of the continuum, which were introduced in \cite{Ban}:
\begin{itemize}
\item[$\mathsf \Delta$] is the smallest weight of an indiscrete finitary coarse structure on $\w$;
\item[$\mathsf \Sigma$] is the smallest weight of an inseparated finitary coarse structure on $\w$.
\end{itemize}
The relation of these two cardinals to other characteristics of the continuum are described in the following theorem, proved in \cite{Ban}.

\begin{theorem}[\cite{Ban}]\label{t:shrek1} $\max\{\mathfrak b,\mathfrak s,\cov(\mathcal N)\}\le\mathsf \Delta\le\mathsf \Sigma\le\non(\M)$.
\end{theorem}

\begin{problem}[\cite{Ban}] Are the strict inequalities $\max\{\mathfrak b,\mathfrak s,\cov(\mathcal N)\}<\mathsf \Delta<\mathsf \Sigma<\non(\M)$ consistent?
\end{problem}

The cardinal $\mathsf\Delta$ and $\mathsf\Sigma$ have dual versions defined as
\begin{itemize}
\item[{}] $\widehat{\mathsf\Delta}=\min\{|\A|:\A\subseteq[\w]^\w\;\wedge\;\forall E\in\IE_{<\w}\;\exists A\in\A\;(A$ is $\{E\}$-discrete$)\}$;
\item[{}] $\widehat{\mathsf\Sigma}=\min\{|\A|:\A\subseteq[\w]^\w\;\wedge\;\forall E\in\IE_{<\w}\;\exists A,B\in\A\;(A,B$ are $\{E\}$-separated$)\}$.
\end{itemize}
For the cardinals  $\mathsf\Delta$ and $\mathsf\Sigma$ we have lower and upper bounds, which are dual to those from Theorem~\ref{t:shrek1}.

\begin{theorem}[\cite{Ban}]\label{t:shrek2} $\cov(\mathcal M)\le\widehat{\mathsf\Sigma}\le\widehat{\mathsf\Delta}\le\min\{\mathfrak d,\mathfrak r,\non(\mathcal N)\}$.
\end{theorem}

\begin{problem}[\cite{Ban}] Are the strict inequalities $\cov(\mathcal M)<\widehat{\mathsf\Sigma}<\widehat{\mathsf\Delta}\le\min\{\mathfrak d,\mathfrak r,\non(\mathcal N)\}$ consistent?
\end{problem}

The cardinal characteristics $\Adelta,\Asigma,\Ddelta,\Dsigma$ nicely fit into the following enriched version of the famous Cicho\'n diagram, see \cite{Ban}.
$$
\xymatrix{
\Asigma\ar[r]&\non(\M)\ar[rr]&&\cof(\M)\ar[r]&\cof(\mathcal N)\ar[ld]\\
&&\mathfrak d\ar[ru]\ar[r]&\mathfrak c&\non(\mathcal N)\ar[u]\\
\Adelta\ar[uu]&\mathfrak s\ar[l]\ar[ru]\ar[rrru]&&\mathfrak r\ar[u]&\Ddelta\ar[u]\ar[l]\ar[llu]\\
\cov(\mathcal N)\ar[u]\ar[rrru]&\w_1\ar[ld]\ar[u]\ar[r]&\mathfrak b\ar[ru]\ar[llu]\ar[luuu]\ar[uu]\\
\add(\mathcal N)\ar[u]\ar[r]&\add(\M)\ar[rr]\ar[ru]&&\cov(\mathcal M)\ar[r]\ar[luuu]&\Dsigma\ar[uu]\\
}
$$

Considering critical cardinalities of {\em cellular} coarse structures on $\w$, we obtain six cardinals:
\vskip3pt

$\mathsf \Delta^\circ_{<\w}=\min\big(\{\mathfrak c^+\}\cup\{w(\E):\E$ is an  indiscrete cellular finitary coarse structure on $\w\}\big)$;

$\mathsf \Sigma^\circ_{<\w}=\min\big(\{\mathfrak c^+\}\cup\{w(\E):\E$ is an inseparated cellular finitary coarse structure on $\w\}\big)$;

$\mathsf \Lambda^\circ_{<\w}=\min\big(\{\mathfrak c^+\}\cup\{w(\E):\E$ is a large cellular finitary coarse structure on $\w\}\big)$;
\smallskip

$\mathsf \Delta^\circ_{\w}=\min\big(\{\mathfrak c^+\}\cup\{w(\E):\E$ is an indiscrete cellular locally finite coarse structure on $\w\}\big)$;

$\mathsf \Sigma^\circ_{\w}=\min\big(\!\{\mathfrak c^+\}\cup\{w(\E)\colon\E$ is an  inseparated cellular locally finite coarse structure on $\w\}\big)$;

$\mathsf \Lambda^\circ_{\w}=\min\big(\{\mathfrak c^+\}\cup\{w(\E):\E$ is a large cellular locally finite coarse structure on $\w\}\big)$.
\smallskip

The cardinal $\mathfrak c^+$ appears in the definitions of those cardinals in order to make these cardinals well-defined (this indeed is necessary for the cardinal $\mathsf \Lambda^\circ_{<\w}$, which is equal to $\mathfrak c^+$ in ZFC).

 The following diagram (taken from \cite{Ban}) describes all known order relations between the cardinals  $\mathsf \Delta^\circ_{<\w}, \mathsf \Sigma^\circ_{<\w}, \mathsf \Lambda^\circ_{<\w}$, $\mathsf \Delta^\circ_{\w}, \mathsf \Sigma^\circ_{\w}, 
\mathsf \Lambda^\circ_{\w}$ and the cardinals $\mathfrak t,\mathfrak b,\mathfrak d,\mathsf \Delta,\mathsf \Sigma,\mathfrak c$. For two cardinals $\alpha,\beta$ an arrow $\alpha\to\beta$ (without label) indicates that $\alpha\le\beta$ in ZFC. A label at an arrow indicates the assumption under which the corresponding inequality holds.
$$
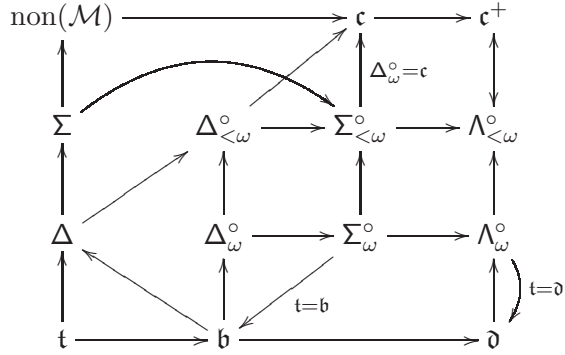
\xymatrix{
\non(\M)\ar[rr]&&\mathfrak c\ar[r]&\mathfrak c^+\\
\Asigma\ar[u]\ar@/^25pt/[rr]&\Adelta_{<\w}^\circ\ar[r]\ar[ur]&\Asigma_{<\w}^\circ\ar[r]\ar_{\Adelta_\w^\circ=\mathfrak c}[u]&\Alambda_{<\w}^\circ\ar@{<->}[u]\\
\Adelta\ar[ur]\ar[u]&\Adelta_{\w}^\circ\ar[r]\ar[u]&\Asigma_{\w}^\circ\ar[r]\ar[u]\ar^{\mathfrak t=\mathfrak b}[ld]&\Alambda_{\w}^\circ\ar[u]\ar@/^10pt/^{\mathfrak t=\mathfrak d}[d]\\
\mathfrak t\ar[r]\ar[u]&\mathfrak b\ar[rr]\ar[u]\ar[ul]&&\mathfrak d\ar[u]\\
}
$$
\medskip

Non-trivial arrows at this diagram follow from Theorem~\ref{t:shrek1} and the following theorem taken from \cite{Ban}.

\begin{theorem}[\cite{Ban}]\label{t:cc}\hfill{\color{white}.}
\begin{enumerate}
\item $\Alambda^\circ_{<\w}=\mathfrak c^+$.
\item $\Adelta_{<\w}^\circ\le\mathfrak c$.
\item $\Adelta^\circ_{<\w}=\mathfrak c$ implies $\Asigma^\circ_{\w}\le\Asigma_\w^\circ=\mathfrak c$.
\item $\mathfrak t=\mathfrak b$ implies $\Adelta_{\w}^\circ=\Asigma_{\w}^\circ=\mathfrak b$. 
\item $\mathfrak t=\mathfrak d$ implies $\Alambda_{\w}^\circ=\mathfrak d$.
\end{enumerate}
\end{theorem}

This theorem raises many set-theoretic questions.

\begin{problem}[\cite{Ban}]\label{prob2.5} Is $\Asigma^\circ_{<\w}\le\mathfrak c$ in ZFC? This is equivalent to asking if there exists an inseparable cellular finitary coarse structure on $\w$ in ZFC?
\end{problem}

Problem~\ref{prob2.5} has afirmative answer if the following problem has negative answer.

\begin{problem}[\cite{Ban}] Is $\Adelta_{<\w}^\circ<\mathfrak c$ consistent?
\end{problem}

\begin{problem}[\cite{Ban}] Does there exist a large cellular locally finite coarse structure on $\w$ in ZFC? If yes, is $\mathfrak d$ equal to the smallest weight of a large cellular locally finite coarse structure on $\w$?
\end{problem}

\begin{problem}[\cite{Ban}] Is $\mathfrak b$ equal to the smallest weight of an indiscrete  cellular locally finite coarse structure on $\w$?
\end{problem}

Two coarse structures $\E,\E'$ on a set $X$ are called 
\begin{itemize}
\item {\em $\delta$-equivalent} if two subsets $A,B\subseteq X$ are $\E$-close if and only if they are $\E'$-close;
\item {\em $\lambda$-equivalent} if two subsets $A,B\subseteq X$ are $\E$-separated if and only if they are $\E'$-separated.
\end{itemize}
These equivalences were introduced and studied in \cite{b16} where the following problems is posed. 

\begin{problem}[{\cite{b16}}]\label{prob:delta} Are any two $\delta$-equivalent finitary coarse stuctures on $\w$ equal?
\end{problem}

\begin{remark} By \cite[6.12]{Ban}, under $\mathfrak b=\mathfrak c$ there exists continuum many large finitary coarse structures on $\w$. All of them are $\delta$-equivalent, which means that Problem~\ref{prob:delta} has negative answer under $\mathfrak b=\mathfrak c$. So we are actually interested in the ZFC-answer to Problem~\ref{prob:delta}.
\end{remark}


\begin{remark} By \cite{Prot2020}, any metrizable unbounded locally finite coarse structure $\E$ on $\w$ is $\lambda$-equivalent to some finitary cellular coarse structure $\E_0\subseteq \E$. 
\end{remark}

It can be shown that two metrizable coarse structures on the same set are equal if and only if they are $\delta$-equivalent.

\begin{problem}[\cite{P-Dynamics}, \cite{b16}]  Let $\E,\E'$ be two $\delta$-equivalent finitary coarse structures on $\w$. Is $\E$ metrizable if $\E'$ is metrizable?
\end{problem}

Finally we discuss some problems related to cardinal characteristics of the family $\IE^\bullet_{<\w}$ of nontrivial cellular finitary entourages on $\w$. We are interested in two cardinal characteristics of the family $\IE^\bullet_{<\w}$:
\begin{itemize}
\item[] ${\uparrow\!\downarrow}(\IE^\bullet_{<\w}):=\min\{|\A|:\A\subseteq\IE^\bullet_{<\w}\;\;\forall E\in \IE^\bullet_{<\w}\;\exists E'\in\IE^\bullet_{<\w}\;\exists A\in\A\;\;(E'\subseteq E\cap A)\}$;
\item[] ${\downarrow\!\uparrow}(\IE^\bullet_{<\w}):=\min\{|\A|:\A\subseteq\IE^\bullet_{<\w}\;\;\forall E\in \IE^\bullet_{<\w}\;\exists E'\in\IE^\bullet_{<\w}\;\exists A\in\A\;\;(E\cup A\subseteq E')\}$.
\end{itemize}
The cardinal characteristics ${\uparrow\!\downarrow}(\IE^\bullet_{<\w})$ and ${\downarrow\!\uparrow}(\IE^\bullet_{<\w})$ are introduced and studied in \cite[\S7]{Ban}, where the following theorem is proved.

\begin{theorem}[{\cite[7.1]{Ban}}] $\Asigma\le {\uparrow\!\downarrow}(\IE^\bullet_{<\w})\le\non(\M)$ and $\cov(\mathcal M)\le {\downarrow\!\uparrow}(\IE^\bullet_{<\w})\le\widehat{\mathsf\Delta}$.
\end{theorem}

This theorem raises the following set-theoretic problems.

\begin{problem} Is $ {\downarrow\!\uparrow}(\IE^\bullet_{<\w})\le\widehat{\mathsf\Sigma}$?
\end{problem}

\begin{problem}\label{prob:poset} Which of the following strict inequalities is consistent?
\begin{enumerate}
\item $\Asigma<{\uparrow\!\downarrow}(\IE^\bullet_{<\w})$;
\item ${\uparrow\!\downarrow}(\IE^\bullet_\w)<\non(\M)$;
\item $\cov(\mathcal M)<{\downarrow\!\uparrow}(\IE^\bullet_{<\w})$;
\item ${\downarrow\!\uparrow}(\IE^\bullet_{<\w})<\widehat{\mathsf\Sigma}$.
\end{enumerate}
\end{problem}

\section{The Higson corona of a coarse space}\label{s:Higson}

Let $(X,\E)$ be a coarse space. A complex-valued function $f:X\to \IC$ is called {\em slowly oscillating} if for any $\e\in(0,\infty)$ and $E\in\E$, there exists an $\E$-bounded set $B\subseteq X$ such that $\mathrm{diam}\,f[E(x)]<\e$ for any $x\in X\setminus B$.

Let $\so(X,\E)$ be the algebra of bounded slowly oscillating complex-valued functions on $X$ and $$\delta:X\to \IC^{\so(X,\E)},\quad \delta:x\mapsto(f(x))_{f\in \so(X,\E)},$$
be the canonical map. Since each function $f\in\so(X,\E)$ is bounded, the image $\delta(X)$ has compact closure $h(X,\E)$, called the {\em Higson compactification} of the coarse space $(X,\E)$.  Since each function $f:X\to \IC$ with finite  support $\supp(f)=\{x\in X:f(x)\ne0\}$ is slowly oscillating, the map $\delta$ is injective and $\delta(X)$ is an open discrete subspace of $h(X,\E)$.

Since every function $f:X\to \IC$ with $\E$-bounded  support is slowly oscillating, for every $\E$-bounded set $B\subseteq X$ the closure $\overline{\delta(B)}$ is an open subset of $h(X,\E)$. Then the complement $$\hbar(X,\E)=h(X,\E)\setminus\bigcup\{\overline{\delta(B)}:\mbox{$B\subseteq X$ is $\E$-bounded in $X$}\}$$is a compact subset of $h(X,\E)$, called the {\em Higson corona} of $(X,\E)$. If each bounded set in $(X,\E)$ is finite, then the Higson corona $\hbar(X,\E)$ coincides with the remainder $h(X,\E)\setminus \delta(X)$ of the Higson compactification of $X$. 

By the {\em Higson corona} of a coarse structure $\E$ on a set $X$ we will understand the Higson corona of the coarse space $(X,\E)$. More information on Higson (and other coronas) can be found in \cite{b5}, \cite{BP-H}, and \cite{Roe}.

Each macro-uniform map $f:X\to Y$ between coarse spaces $X,Y$ induces a continuous map $$hf:h(X)\to h(Y),\;\;hf:(x_\varphi)_{\varphi\in\so(X)}\mapsto (x_{\psi\circ f})_{\psi\in\so(Y)},$$
between their Higson compactifications. If $f$ is a coarse equivalence, then the induced map $\hbar f=hf{\restriction}\hbar(X):\hbar(X)\to\hbar(Y)$ is a homeomorphism of the Higson coronas.

The Higson corona reflects many properties of the coarse space. For example,
\begin{itemize}
\item for any inseparated coarse space its Higson corona is a singleton;
\item for any indiscrete coarse space its Higson corona is finite.
\end{itemize}

Let us recall that a compact Hausdorff space $X$  is {\em perfectly normal} if each closed subset of $X$ is of type $G_\delta$ in $X$. In \cite{BP-H} we studied the problem of reconstruction of a coarse space from its Higson corona and proved the following theorem.

\begin{theorem}[\cite{BP-H}]{\color{white}{\tiny.}}\begin{enumerate}
\item Every perfectly normal compact Hausdorff space is homeomorphic to the Higson corona of some finitary coarse structure on $\w$.
\item Under $\Adelta^\circ_{<\w}=\mathfrak c$, every perfectly normal compact Hausdorff space is homeomorphic to the Higson corona of some cellular finitary coarse structure on $\w$.
\item Under $\w_1=\mathfrak c$, every compact Hausdorff space of weight $\le\w_1$ is homeomorphic to the Higson corona of some cellular finitary coarse structure on $\w$.
\end{enumerate}
\end{theorem}

This theorem raises two open set-theoretic problems.

\begin{problem} Is each compact Hausdorff space of weight $\le\w_1$  homeomorphic to the Higson corona of a finitary coarse structure on $\w$?
\end{problem} 

\begin{problem} Is each compact metrizable space homeomorphic to the Higson corona of some cellular finitary coarse structure on $\w$?
\end{problem} 

 By Theorem 4 in \cite{Prot_bc} (see also Theorem 11 in \cite{DKU}), for any cellular  metrizable finitary coarse space $(X,\E)$ its Higson corona $\hbar(X,\E)$ is a zero-dimensional compact Hausdorff space. Under CH, we have a much more precise result.
 
\begin{theorem}[\cite{Prot11}] Under CH the Higson corona of any unbounded separable ultrametric space is homeomorphic to $\beta\w\setminus\w$.
\end{theorem}

On the other hand, under $\mathfrak u<\mathfrak d$ we have an opposite result.

\begin{theorem}[\cite{BCZ}]  Let $G$ be any countable locally finite group. Under $\mathfrak u<\mathfrak d$, for a metrizable finitary coarse space $(X,\E)$ the following conditions are equivalent: 
\begin{enumerate}
\item $(X,\E)$ is coarsely equivalent to the coarse space $(G,\E_G)$;
\item the Higson coronas $\hbar(X,\E)$ and $\hbar(G,\E_G)$ are homeomorphic;
\item $(X,\E)$ is cellular and the character of each point in $\hbar(X,\E)$ is $\ge\mathfrak d$.
\end{enumerate}
\end{theorem}

We recall that the {\em character} of a point of a topological space is the smallest cardinality of a neighborhood base at the point.

By \cite{SS} (see also \cite{Vel}, \cite{FMcK}), under Proper Forcing Axiom (briefly, PFA), each homeomorphism of the Stone-\v Cech remainder $\beta\w\setminus\w$ is induced by a bijection between  subsets with finite complement in $\w$. This result can be reformulated as follows.

\begin{theorem} Let $\E$ be the smallest coarse structure on the set $X=\w$. Under PFA, each homeomorphism of the corona $\hbar(X,\E)$ is induced by some coarse equivalence of $(X,\E)$.
\end{theorem}

This theorem suggests the following problem.

\begin{problem}\label{pr:PFA} Assume PFA. Let $X,Y$ be two metrizable (cellular) finitary coarse spaces.
\begin{enumerate}
\item Are the coarse spaces $X,Y$ coarsely equivalent if their Higson coronas are homeomorphic? 
\item Is any homeomorphism between the Higson coronas of $X,Y$ induced by a coarse equivalence between $X$ and $Y$?
\end{enumerate}
\end{problem} 

\begin{remark} Recent results of Braga, Farah, Vignati \cite{BFV} suggest that the answer to Problem~\ref{pr:PFA} can be affirmative for metrizable finitary coarse spaces   $X,Y$ of finite asymptotic dimension (more generally, for  metrizable finitary coarse spaces having Yu's property $A$). On the other hand, by an old result of Rudin \cite{Rudin}, under CH the Stone-\v Cech remainder has $2^{\mathfrak c}$ homeomorphisms, which implies that under CH the answer to Problem~\ref{pr:PFA}(2) is negative even for discrete coarse spaces.
\end{remark}


\section{Special ultrafilters on coarse spaces}\label{s:ultra}

First we recall the definitions of some special types of free ultrafilters on $\w$. 

A free ultrafilter $\U$ on $\w$ is called
\begin{itemize}
\item a {\em $P$-point} if for any sequence of sets $\{U_n\}_{n\in\w}\subseteq\U$ there exist a set $U\in\U$ such that $U\subseteq^* U_n$ for every $n\in\w$;
\item a {\em $Q$-point} if for any locally finite cellular entourage $E$ on $\w$ there exists a set $U\in\U$ such that $|U\cap E(x)|\le 1$ for any $x\in\w$;
\item {\em Ramsey} if for any map $f:\w\to\w$ there exists $U\in\mathcal U$ such that $f{\restriction}U$ is either constant or injective;
\item a {\em weak $P$-point}, if for any sequence $(\U_n)_{n\in\w}$ of free ultrafilters on $\w$ that are not equal to $\U$ there exists a set $U\in\U\setminus\bigcup_{n\in\w}\U_n$;
\item an {\em $OK_\kappa$-point} for a cardinal $\kappa$ if for any sequence of sets $\{U_n\}_{n\in\w}\subseteq \U$ there exists a family $(V_\alpha)_{\alpha\in\kappa}\subseteq\U$ such that for any ordinals $\alpha_1<\dots<\alpha_n$ in $\kappa$ we have $\bigcap_{i=1}^nV_{\alpha_i}\subseteq^*U_n$;
\item {\em rapid} if for any function $f\in\w^\w$ there exists a function $g\in \w^\w$ such that $f\le g$ and $\{g(n):n\in\w\}\in\U$;
\item {\em discrete} if for any injective function $\varphi:\w\to \IR$ there exists a set $U\in\U$ whose image $\varphi(U)$ is a  discrete subspace of $\IR$;
\item {\em nowhere dense} if for any injective function $\varphi:\w\to \IR$ there exists a set $U\in\U$ whose image $\varphi(U)$ is nowhere dense in $\IR$;
\item a ({\em cellular}) {\em $T$-point} if for any increasing sequence $(E_n)_{n\in\w}$ of finitary (and cellular) entourages on $\w$ there exists a set $U\in\U$ such that for every $n\in\w$ the set $\{x\in U:\{x\}\ne U\cap E_n(x)\}$ is finite;
\item {\em dynamically discrete} if for any action of a countable group $G$ on $\w$ the subspace $\{g\,\U\}_{g\in G}$ is discrete in $\beta\w$.
\end{itemize}
It is known that a free ultrafilter on $\w$ is Ramsey if and only if it is both a $P$-point and a $Q$-point. More information on $Q$-points, $P$-points, weak $P$-points, $OK_\kappa$-points can be found in the survey \cite{vM}; discrete and nowhere dense ultrafilters were considered in \cite{BB} and \cite{Shelah}; (cellular) $T$-points were introduced and studied by Petrenko and Protasov in \cite{b3}, \cite{b4}.

The following diagram describes the implications between various properties of free ultrafilters on $\w$.
$$
\xymatrix{
\mbox{rapid}&\mbox{$Q$-point}\ar@{=>}[l]\ar@{=>}[r]&\mbox{$T$-point}\ar@{=>}[r]&\mbox{cellular $T$-point}\\
&\mbox{Ramsey}\ar@{=>}[r]\ar@{=>}[u]&\mbox{$P$-point}\ar@{=>}[r]\ar@{=>}[u]\ar@{=>}[d]&\mbox{discrete}\ar@{=>}[r]&\mbox{nowhere}\atop\mbox{dense}\\
\mbox{$OK_{\mathfrak c}$-point}\ar@{=>}[r]&\mbox{$OK_{\w_1}$-point}\ar@{=>}[r]&\mbox{weak $P$-point}\ar@{=>}[r]&\mbox{dynamically}\atop\mbox{discrete}
}
$$

According to a famous result of Kunen \cite{Kunen76}, \cite{Kunen} (see also \cite[4.5.2]{vM}), $OK_{\mathfrak c}$-points exist in ZFC. On the other hand, there are models of ZFC containing no nowhere dense ultrafilters \cite{Shelah} and there are models of ZFC containing no rapid ultrafilters \cite{Miller}. A well-known open problem \cite[p.563]{vM} asks whether there exists a model of ZFC containing no $P$-points and no $Q$-points.

This information motivates the following open problems.

\begin{problem}[\cite{b3}] Do $T$-points exist in ZFC?
\end{problem}

\begin{problem}[\cite{b3}] Is each weak $P$-point (or $OK_{\mathfrak c}$-point) a T-point?
\end{problem}

\begin{problem}[\cite{b4}] Is each discrete ultrafilter a $T$-point?
\end{problem}

\begin{problem}[\cite{b4}] Is each cellular $T$-point a $T$-point?
\end{problem}

\begin{remark} By \cite[Proposition 4]{b3}, under CH there exists a $T$-point     
which is neither a weak $P$-point, nor a $Q$-point  nor a nowhere dense ultrafilter.
Jana Fla\v skova noticed that a rapid ultrafilter needs not be a $T$-point, see \cite[p.350]{b4}.
\end{remark}

It is easy to see that weak $P$-points are dynamically discrete ultrafilters. The converse is not true because of the following result. Two ultrafilters $\U,\V$ on $\w$ are called {\em isomorphic} if there exists a bijection $f$ of $\w$ such that $\U=\{f(V):V\in\V\}$. By \cite[4.5.2]{vM}, there exists $2^{\mathfrak c}$ pairwise non-isomorphic weak $P$-points.

\begin{theorem} If $(\U_n)_{n\in\w}$ is a sequence of pairwise non-isomorphic weak $P$-points, then each ultrafilter $\W$ in the closure $\overline{\{\U_n:n\in\w\}}\subseteq \beta\w\setminus\w$ is dynamically discrete.
\end{theorem}

\begin{proof} Since $\W\in\overline{\{\U_n:n\in\w\}}$, there exists an ultrafilter $\V$ on $\w$ such that the family $$\textstyle\big\{\bigcup_{n\in V}U_n:V\in\V\;\wedge\;(U_n)_{n\in V}\in\prod_{n\in V}\U_n\big\}$$is a base of the ultrafilter $\W$. If the ultrafilter $\V$ is principal, then $\W=\U_n$ for some $n\in\w$ and $\W=\U_n$ is dynamically discrete, being a weak $P$-point. So, we assume that $\V$ is a free ultrafilter.  Let $G$ be any countable subgroup of the permutation group $S_\w$ of $\w$. Write $G$ as the union $G=\bigcup_{n\in\w}G_n$ of a sequence $(G_n)_{n\in\w}$ of finite sets such that $G_n^{-1}=G_n\subseteq G_{n+1}$ for every $n\in\w$.

For every $n\in\w$  the weak $P$-point $\U_n$ does not belong to the closure of the countable set $\{g\U_i:g\in G,\;i\in\w,\;\U_n\ne g\U_i\}$. Consequently, there exists a set $U_n\in\U_n$ such that $U_n\notin g\U_i$ for every $g\in G$ and $i\in\omega$ with $\U_n\ne g\U_i$. It follows that for every $n\in\w$ the set $$U'_n=U_n\setminus\bigcup\{gU_i:g\in G_n,\;i\le n,\; \U_n\ne g\U_i\}$$belongs to the ultrafilter $\U_n$.

To see $\W$ is dynamically discrete, it suffices to show that the ultrafilter $\W$ is an isolated point of its orbit $\{g\W\}_{g\in G}$. This will follow as soon as we check that for every $g\in G$ with $g\W\ne\W$, the set $W=\bigcup_{n\in\w}U'_n\in\W$ does not belong to the ultrafilter $g\W$. Find a number $m\in\w$ such that $g\in G_m$. Since $\W\ne g\W$, we can chose a set $O_g\in\W$ such that $O_g\cap gO_g=\emptyset$. The choice of $\V$ ensures that the set $V_g=\{n\in\w:O_g\in \U_n,\;n\ge m\}$  belongs to the ultrafilter $\V$. We claim that the set $W_g:=\bigcup_{n\in V_g}gU'_n\in g\W$ is disjoint with the set $W$. Assuming that $W\cap W_g\ne\emptyset$, we could find two numbers $i\in\w$ and $j\in V_g$ such that $U'_i\cap gU'_j\ne\emptyset$.
If $i=j$, then the inclusion $i\in V_g$ implies $\U_i\ne g\U_i$ and then the definition of the set $U'_i$ ensures that $U'_i\cap gU'_i=\emptyset$, which contradicts the choice of the numbers $j,i$. Therefore, $i\ne j$.

The non-isomorphness of the ultrafilters $\U_i,\U_j$ guarantees that $g\U_j\ne \U_i\ne g^{-1}\U_j$. If $j<i$, then the definition of the set $U'_i$ ensures that $U'_i\cap gU'_j=\emptyset$. If $i<j$, then the definition of the set $U'_j$ ensures that  $U'_j\cap g^{-1}U'_i=\emptyset$. In both cases we obtain $U'_i\cap gU'_j=\emptyset$, which contradicts the choice of the numbers $i,j$.
This completes the proof of dynamical discreteness of $\W$.
\end{proof}

\section{Coarse structures and topological dynamics}\label{s:d}

By Theorem~\ref{t:G}, every finitary coarse structure on a set $X$ is equal to the finitary coarse structure $\E_G$ induced by the action of suitable subgroup $G\subseteq S_X$ of the permutation group of $X$. The action of the group $G$ extends to a continuous action of $G$ on the Stone-\v Cech compactification $\beta X$ of the discrete space $X$. The  remainder $X^*=\beta X\setminus X$ is an invariant subset of the action, so we obtain a dynamical system $(X^*,G)$ and can study the interplay between properties of the coarse structure $\E_G$ and the properties of the dynamical system $(X^*,G)$, see the papers \cite{b7}, \cite{P-Dynamics}, \cite{PS15} for more information on this topic.

The following theorem (proved in \cite[Theorem 15]{b16})  shows that the orbit structure of the dynamical system $(\w^*,G)$ uniquely determines the coarse structure $\E_G$.

\begin{theorem} For two subgroups $G,H\subseteq S_\w$ the coarse structures $\E_G$ and $\E_H$ on $\w$ coincide if and only if $\{Gp:p\in\w^*\}=\{Hp:p\in\w^*\}$.
\end{theorem} 

The following fact is proved in \cite[Theorem 3.15]{P-Dynamics}.

\begin{theorem} If for a group $G$ of permutations of a set $X$,  the finitary coarse space $(X,\E_G)$ is indiscrete,  then there exists an ultrafilter $p\in X^*$ whose orbit $Gp$ is not discrete.
\end{theorem}

Under $\mathfrak t=\mathfrak c$ we can prove a bit more.

\begin{theorem}\label{t:t=c} Let $G$ be a group of permutations of a countable set $X$ such that the coarse space $(X,\E_G)$ is indiscrete. Then there exists an ultrafilter $p\in X^*$ whose orbit $Gp$ contains at least $\mathfrak t$ pairwise disjoint nonempty open sets and hence $|Gp|\ge\mathfrak t$. If $\mathfrak t=\mathfrak c$, then $p$ is a $P$-point and the space $Gp$ is not discrete but every subspace $A\subseteq Gp$ of cardinality $|A|<\mathfrak c$ is discrete.
\end{theorem}

\begin{proof} Write the family of all subsets of $X$ as $\{X_\alpha\}_{\alpha\in\mathfrak c}$.

For every ordinal $\alpha\in\mathfrak t$ we shall inductively choose an infinite set $U_\alpha\subseteq X$ and an element $g_\alpha\in G$ such that the following conditions are satisfied:
\begin{itemize}
\item[(i)] $U_\alpha\cup g_\alpha U_\alpha\subseteq^* U_\beta$ for all $\beta\in\alpha$;
\item[(ii)] $U_\alpha\subseteq X_\alpha$ or $U_\alpha\subseteq X\setminus X_\alpha$;
\item[(iii)] $U_\alpha\cap g_\alpha U_\alpha=\emptyset$.
\end{itemize}
Assume that for some ordinal $\alpha\in \mathfrak t$ we have constructed a family of infinite sets $(U_\beta)_{\beta\in\alpha}$ satisfying the condition (i). Since $\alpha<\mathfrak t$, there exists an infinite set $W_\alpha\subseteq X$ such that $W_\alpha\subseteq^* U_\beta$ for all $\beta\in\alpha$. Since one of the sets $W_\alpha\cap X_\alpha$ or $W_\alpha\setminus X_\alpha$ is infinite, we can replace $W_\alpha$ by a smaller infinite set and additionally assume that $W_\alpha\subseteq X_\alpha$ or $W_\alpha\subseteq X\setminus X_\alpha$.

Since the coarse space $(\w,\E_G)$ is indiscrete, the subset $W_\alpha$ is not $\E_G$-discrete. Consequently, there exists an element $g_\alpha\in G$ such that the set $V_\alpha=\{x\in W_\alpha:x\ne g_\alpha x\in W_\alpha\}$ is infinite. Choose an infinite subset $U_\alpha\subseteq V_\alpha$ such that $U_\alpha\cap g_\alpha U_\alpha=\emptyset$. It is easy to see that the set $U_\alpha$ satisfies the conditions (i)--(iii).

After completing the inductive construction, take any ultrafilter $p\in X^*$ containing the family $\{U_\alpha\}_{\alpha\in\mathfrak t}$. The inductive conditions (i) and (iii) imply that for every ordinals $\beta<\alpha<\mathfrak t$ we have
$$g_\alpha U_\alpha\cap g_\beta U_\beta\subseteq^* U_\beta\cap (X\setminus U_\beta)=\emptyset,$$
which implies that the family $(Gp\cap \overline{g_\alpha U_\alpha})_{\alpha\in\mathfrak t}$ consists of pairwise disjoint nonempty open sets in $Gp\subseteq X^*\subset \beta X$.

Now assume that $\mathfrak t=\mathfrak c$. In this case the condition (ii) implies that $\{U_\alpha\}_{\alpha\in\mathfrak t}$ is a base of the ultrafilter $p$. The condition (i) implies that the ultrafilter $p$ is a $P$-point (even a $P_{<\mathfrak c}$-point). To see that the orbit $Gp$ is not discrete, take any neighborhood $O_p$ of $p$ in $\beta X$ and find $\alpha\in\mathfrak t$ such that $\overline U_{\!\alpha}\subseteq O_p$. The inductive condition (i) implies that $g_{\alpha+1}U_{\alpha+1}\subseteq^* U_\alpha$ and hence $g_{\alpha+1}p\in \overline U_{\alpha}\subseteq O_p$, witnessing that the orbit $Gp$ is not discrete.

It remains to prove that every subspace $A\subseteq Gp$ of cardinality $|A|<\mathfrak c$ is discrete. Given any point $gp\in A$, for every $a\in A\setminus\{gp\}$, find an ordinal $\alpha_a\in\mathfrak c$ such that $gU_{\alpha_a}\notin a$. By Proposition 6.4 in \cite{Blass}, the cardinal $\mathfrak t$ is regular. Consequently, there exists an ordinal $\alpha\in\mathfrak c=\mathfrak t$ such that $\sup_{a\in A}\alpha_a<\alpha$. The inductive condition (i) guarantees that $gU_\alpha\subseteq^* gU_{\alpha_a}$ for all $a\in A\setminus\{gp\}$. Then $g\overline U_{\!\alpha}$ is a neighborhood of $gp$ in $\beta X$, which is disjoint with the set $A\setminus\{gp\}$ and witnesses that the space $A$ is discrete.
\end{proof}

By a {\em dynamical system} we understand a compact Hausdorff space $K$ endowed with the continuous action of some group $G$. A dynamical system $(K,G)$ is called
\begin{itemize}
\item {\em minimal} if the orbit $Gx$ of any point $x\in K$ is dense in $K$;
\item {\em topologically transitive} if the orbit $GU$ of any nonempty open set $U\subseteq K$ is dense in $K$.
\end{itemize}

The following theorem proved in \cite{P-Dynamics}  characterizes large and inseparated coarse structures in dynamical terms.

\begin{theorem}[\cite{P-Dynamics}]\label{t:dyn1} For an infinite set $X$ and a subgroup $G\subseteq S_X$ of the permutation group, the coarse structure $\E_G$ is
\begin{enumerate}
\item large if and only if the dynamical system $(X^*,G)$ is minimal;
\item inseparated if and only if $(X^*,G)$ is topologically transitive.
\end{enumerate}
\end{theorem}

It is well-known that a dynamical system $(K,G)$ of a (metrizable) compact Hausdorff space $K$ is topologically transitive if (and only if) the orbit $Gx$ of some point $x\in K$ is dense in $K$. In this case we say that the dynamical system {\em has a dense orbit}.

\begin{example}\label{ex:indep} Under $\Asigma<\mathfrak c$ there exists a subgroup  $G\subseteq S_\w$ of cardinality $|G|=\Asigma<\mathfrak c$ such that the coarse structure $\E_G$ is inseparated and hence the dynamical system $(\w^*,G)$ is topologically transitive but has no dense orbits (as the space $\w^*$ has density $\mathfrak c>|G|$).
\end{example}

\begin{proof} By the definition of the cardinal $\mathsf \Sigma$, there exists an inseparated finitary coarse structure $\mathcal E$ of weight $w(\mathcal E)=\mathsf\Sigma$ on $\omega$. By Theorem~\ref{t:G}, the coarse structure $\mathcal E$ coincides with the coarse structure $\mathcal E_G$ generated by some subgroup $G\subseteq S_\omega$ that has cardinality $|G|=w(\mathcal E)=\mathsf\Sigma$. Since the coarse space $\mathcal E=\mathcal E_G$ is inseparated, we can apply Theorem~\ref{t:dyn1}(2) and conclude that the dynamical system $(\omega^*,G)$ is topologically transitive. Since the Stone-\v Cech remainder $\omega^*$ contains continuum many pairwise disjoint open sets, its density is equal to the cardinality of continuous. Since $|G|=\mathsf \Sigma<\mathfrak c$, no orbit $Gp$ is dense in $\omega^*$.
\end{proof}

Example~\ref{ex:indep} cannot be proved in ZFC because of the following theorem. 

\begin{theorem}\label{t=c} Under $\mathfrak t=\mathfrak c$, for every subgroup $H$ of the homeomorphism group of $\w^*$, the dynamical system $(\w^*,H)$ is topologically transitive if and only if it has a dense orbit. 
\end{theorem}

\begin{proof} The ``if'' part is trivial and holds without any set-theoretic assumptions. To prove the ``only if'' part, assume that $\mathfrak t=\mathfrak c$ and the dynamical system $(\w^*,H)$ is topologically transitive.

Let $(A_\alpha)_{\alpha\in\mathfrak c}$ be an enumeration of all infinite subsets of $\omega$. By transfinite induction we shall construct a transfinite sequence of infinite subsets $(U_\alpha)_{\alpha\in\mathfrak c}$ of $\omega$ and a transfinite sequence $(g_\alpha)_{\alpha\in\mathfrak c}$ of elements of the group $H$ such that for every $\alpha\in\mathfrak c$ the following conditions are satisfied:
\begin{itemize}
\item[(a)] $U_\alpha\subseteq^* U_\beta$ for all $\beta<\alpha$;
\item[(b)] $g_\alpha(U_\alpha)\subseteq^* A_\alpha$.
\end{itemize}

To start the inductive construction, put $U_0=A_0$ and $g_0$ be the identity of the group $H$. Assume that for some ordinal $\alpha\in\mathfrak c$, a transfinite sequence $(U_\beta)_{\beta<\alpha}$ satisfying the condition (a) has been constructed. By the definition of the tower number $\mathfrak t$ and the equality $\mathfrak t=\mathfrak c>\alpha$, there exists an infinite subset $V_\alpha\subseteq\omega$ such that $V_\alpha\subseteq^* U_\beta$ for all $\beta<\alpha$. The infinite sets $V_\alpha$ and $A_\alpha$ determine basic open sets $V^*_\alpha$ and $A^*_\alpha$ in the space $\omega^*$. Since the action of the group $H$ on $\omega^*$ is topologically transitive, there exist $g_\alpha\in H$ and an infinite subset $U_\alpha\subseteq V_\alpha$ such that $g_\alpha(U^*_\alpha)\subseteq A^*_\alpha$ and hence $g_\alpha(U_\alpha)\subseteq^* A_\alpha$. This completes the inductive step.

After completing the inductive construction, extend the family $\{U_\alpha\}_{\alpha\in\mathfrak c}$ to a free ultrafilter $\mathcal U$ and observe that its orbit intersects each basic open set $A^*_\alpha$, $\alpha\in\mathfrak c$ and hence is dense in $\omega^*$. 
\end{proof}

Theorem~\ref{t:dyn1} and the (original) definition of the cardinal $\Asigma$ given in \cite{Ban} imply the following dynamical characterization of $\Asigma$.

\begin{theorem}\label{t:char-Sigma} The cardinal $\Asigma$ is equal to the smallest cardinality of a subgroup $G\subseteq S_\w$ such that the dynamical system $(\w^*,G)$ is topologically transitive.
\end{theorem}

Theorems~\ref{t=c} and \ref{t:char-Sigma} suggest to consider the cardinal $\Asigma^*$ defined as the smallest cardinality of a subgroup $H\subseteq\Homeo(\w^*)$ such that the dynamical system $(\w^*,H)$ is topologically transitive.

\begin{theorem} $\mathfrak t\le\Asigma^*\le\Asigma$.
\end{theorem}

\begin{proof} The inequality $\Asigma^*\le\Asigma$ follows from Theorem~\ref{t:char-Sigma}. To prove that $\mathfrak t\le \Asigma^*$, we need to show that a subgroup $H\subseteq \Homeo(\w^*)$ has cardinality $|H|\ge \mathfrak t$ if the dynamical system $(\w^*,H)$ is topologically transitive.

By \cite[8.1]{Blass}, there exists a family $(A_\alpha)_{\alpha\in\mathfrak c}$ of infinite subsets of $\w$ such that for any distinct ordinals $\alpha,\beta\in\mathfrak c$ the intersection $A_\alpha\cap A_\beta$ is finite. Repeating the argument from the proof of Theorem~\ref{t=c}, we can use the topological transitivity of the dynamical system $(\w^*,H)$ and construct a transfinite sequence of infinite subsets $(U_\alpha)_{\alpha\in\mathfrak t}$ of $\omega$ and a transfinite sequence $(g_\alpha)_{\alpha\in\mathfrak t}$ of elements of the group $H$ such that for every $\alpha\in\mathfrak t$ the following conditions are satisfied:
\begin{itemize}
\item[(a)] $U_\alpha\subseteq^* U_\beta$ for all $\beta<\alpha$;
\item[(b)] $g_\alpha(U_\alpha)\subseteq^* A_\alpha$.
\end{itemize}
We claim that the sequence $(g_\alpha)_{\alpha\in\mathfrak t}$ consists of pairwise distinct elements of the group $H$. To derive a contradiction, assume that $g_\alpha= g_\beta$ for some ordinals $\alpha<\beta$ in $\mathfrak t$. Then $$g_\beta(U_\beta)=g_\alpha(U_\beta)\cap g_\beta(U_\beta)\subseteq^* g_\alpha(U_\alpha)\cap g_\beta(U_\beta)\subseteq A_\alpha\cap A_\beta$$and hence the set $g_\beta(U_\beta)$ is finite and so is the set $U_\beta$, which contradicts the choice of $U_\beta$. This contradiction shows that $|H|\ge|\{g_\alpha\}_{\alpha\in\mathfrak t}|=\mathfrak t$ and hence $\Asigma^*\ge \mathfrak t$.
\end{proof}

\begin{problem} Is the strict inequality $\Asigma^*<\Asigma$ consistent?
\end{problem}

A point $x$ of a topological space $X$ is called 
\begin{itemize}
\item a {\em $P$-point} if every $G_\delta$-set $G\subseteq X$ that contains $x$ is a neighborhood of $x$;
\item a {\em weak $P$-point} if $x\notin\overline C$ for any countable set $C\subseteq X\setminus\{x\}$.
\end{itemize}

For any coarse space $(X,\E)$, the canonical map  $\delta:X\to h(X,\E)$ from $X$ to its Higson compactification has a unique continuous extension $\bar\delta:\beta X\to h(X,\E)$ to the Stone-\v Cech compactification $\beta X$ of $X$ endowed with the discrete topology. If the coarse structure $\E$ is locally finite, then $\bar\delta(\beta X\setminus X)=\hbar(X,\E)$. For a free ultrafilter $p$ on $X$ let $\check p$ be the set $\bar\delta^{-1}(\bar\delta(p))$. For two free ultrafilters $p,q$ on $X$ we have $\check p=\check q$ if and only if for any bounded slowly oscillating function $f:X\to\IC$ and its continuous extension $\bar f:\beta X\to\IC$ we have $\bar f(p)=\bar f(q)$. 

Our next questions ask about the interplay between properties of points in $\beta X\setminus X$ and $\hbar(X,\E)$.
By Theorem 3 \cite{b6}, for a metrizable finitary coarse structure $\E$ on $\w$ the Higson corona $\hbar(\omega, \mathcal{E})$  contains  a weak $P$-point. Moreover, $\hbar(\w,\E)$ contains a $P$-point if $\beta\w\setminus\w$ contains a  $P$-point.

\begin{problem}[\cite{b6}] Let $\mathcal{E}$  be a finitary metrizable  coarse structure on $\omega$ and $p$ be a free ultrafilter on $\w$.
\begin{itemize}
\item[(i)] Is $\bar\delta(p)$ a weak  $P$-point in $\hbar(\omega,\mathcal{E})$ if  $p$ is a weak $P$-point?
\item[(ii)] Does the set $\check p$ contain a weak  $P$-point in $\beta\w\setminus \w$ if $\bar\delta(p)$ is a weak $P$-point in $\hbar(\w,\E)$?
\end{itemize} 
\end{problem}

Given a countable subgroup $G\subseteq S_\w$, consider the finitary coarse structure $\E_G$ on $\w$ induced by the action of the group $G$. 
 By analogy with Theorem 4.1 \cite{b7}, it can be shown that for any $P$-point $p\in\w^*$ the set $\check p$ coincides with the closure $\overline{Gp}$ of the orbit of $p$ under the action of the group $G$ on $\w^*$. 

\begin{problem}[\cite{b7}] Is a free ultrafilter $p$ on $\w$ a $P$-point if $\check p=\overline{Gp}$ for any countable subgroup $G\subseteq S_\w$?
\end{problem}

If an answer to this question is negative, then another question is of interest.

\begin{problem}[\cite{b7}]  In ZFC, does there exists a free ultrafilter $p$ on $\w$ such that $\check{p}=\overline{Gp}$ for any countable subgroup $G\subseteq S_\w$?
\end{problem}

\section{Partitions of coarse spaces into thin subsets}\label{s6}

Let $(X,\E)$ be a coarse space and $m$ be a natural number. A subset $T\subseteq X$ is called {\em $m$-thin} if for any entourage $E\in\E$ there exists an $\E$-bounded set $B\subseteq X$ such that $|E(x)\cap T|\le m$ for every $x\in T\setminus B$. It is clear that $T\subseteq X$ is $1$-thin if and only if $T$ is $\E$-discrete; $1$-thin sets are called {\em thin}. 

The following dynamical characterization of $m$-thin sets was proved in \cite[Proposition~4]{PS15} (for $m=1$ this characterization was proved in Theorem 3.2 in  \cite{PS14}).

\begin{theorem} Let $m\in\mathbb N$ and $G\subseteq S_\w$ be a subgroup. A subset $T\subseteq\w$ of the coarse space $(\w,\E_G)$ is $m$-thin if and only if $|\overline T\cap Gp|\le m$ for any ultrafilter $p\in\w^*$. Here $\overline T$ is the closure of the set $T$ in $\beta\w$. 
\end{theorem}

This theorem implies  the following characterization of decomposability into finitely many thin pieces.

\begin{theorem} For a subgroup $G\subseteq S_\w$, a subset $T\subseteq\w$ of the coarse space $(\w,\E_G)$ can be covered by finitely many thin sets if and only if any ultrafilter $p\in \overline T\cap\w^*$ contains a set $P\in p$ such that $|\overline  P\cap Gq|\le 1$ for every $q\in \w^*$. 
\end{theorem}

By \cite{b15}, for every $n,m\in\mathbb N$ and Abelian group $G$ of cardinality $\aleph _n$, every $m$-thin subset $T\subseteq G$ in the coarse space $(G,\E_G)$ is the union of  $m^{n+1}$ thin subsets. On the other hand, there exists a group $G$  of cardinality  $\aleph _{\omega}$ such that the finitary coarse space $(G,\E_G)$ contains a 2-thin subset which cannot be covered  by finitely many  thin subsets. 
In fact, similar examples can be found among finitary coarse spaces of countable cardinality. 

A group $G$ is called {\em Boolean} if each element $x\in G$ has order $\le 2$.  A subgroup $G\subseteq S_\w$ is defined to be {\em almost Boolean} if the quotient group $G/(G\cap S_{<\w})$ is Boolean. In this definition $S_{<\w}$ denotes the normal subgroup of $S_\w$ consisting of permutations $f\in S_\w$ that have finite support $\supp(f)=\{x\in\w:f(x)\ne x\}$. 

\begin{theorem}\label{t:thin} There exists an almost Boolean subgroup $G\subseteq S_\w$ of cardinality $|G|=\mathfrak r$ such that the finitary coarse space $(\w,\E_G)$ is cellular and $2$-thin but $\w$ cannot be covered by finitely many $1$-thin subspaces.
\end{theorem}

\begin{proof} We divide the proof of this theorem into three lemmas.

\begin{lemma}\label{l:thin=c} There exists an almost Boolean subgroup $G\subseteq S_\w$ of cardinality $|G|=\mathfrak c$ such that the finitary coarse space $(\w,\E_G)$ is cellular and $2$-thin but $\w$ cannot be covered by finitely many $1$-thin subspaces.
\end{lemma}

\begin{proof} By \cite[8.1]{Blass}, there exists a family $\A\subseteq[\w]^\w$ of cardinality $|\A|=\mathfrak c$, which is {\em almost disjoint} in the sense that $A\cap B$ if finite for any distinct sets $A,B\in\A$.
 
Fix any ultrafilter $\mathcal U$ on $\w$ containing the filter $$\mathcal F=\{X\in[\w]^\w:|\{A\in\A:A\not\subseteq^* X\}|<\mathfrak c\}.$$

Let $\{U_\alpha\}_{\alpha\in\mathfrak c}$ be an enumeration of the ultrafilter $\U$. By transfinite induction, for every $\alpha\in\mathfrak c$ choose a set $A_\alpha\in\A\setminus\{A_\beta\}_{\beta\in\alpha}$ such that the set $A_\alpha\cap U_\alpha$ is infinite. The choice of $A_\alpha$ is always possible since the set $\{A\in\A:|A\cap U_\alpha|=\w\}$ has cardinality of continuum (in the opposite case, the set $\w\setminus U_\alpha$  belongs to the filter $\F\subseteq\U$ but has empty intersection with $U_\alpha$, which is a desired contradiction).

After completing the inductive construction, for every $\alpha\in\mathfrak c$, choose an involution $f_\alpha$ of $\w$ whose support $\supp(f_\alpha):=\{x\in\w:f(x)\ne x\}$ is an infinite subsets of $U_\alpha\cap A_\alpha$. Let $G\subseteq S_\w$ be the subgroup, generated by the set $\{f_\alpha\}_{\alpha\in\mathfrak c}$. Taking into account that the family $(\supp(f_\alpha))_{\alpha\in\mathfrak c}$ is almost disjoint, one can show that the group $G$ is almost Boolean and the coarse space $(\w,\E_G)$ is cellular and $2$-thin. 

To see that $(\w,\E_G)$ cannot be covered by finitely many thin sets, it suffices to observe that for every finite partition of $\w$ one of the cells of the partition belongs to the ultrafilter $\U$, and that for every $\alpha\in\mathfrak c$ the set $U_\alpha\in\U$ is not $1$-thin since the set $\{x\in U_\alpha:f_\alpha(x)\ne x\}$ is infinite.
\end{proof} 

For treating the case $\mathfrak r<\mathfrak c$, we shall exploit the following lemma, which is a special case of Theorem 2.1 in \cite{BHM}.

\begin{lemma}\label{l:a'} For any cardinal $\kappa<\mathfrak c$ and family of infinite sets $\{I_\alpha\}_{\alpha\in\kappa}\subseteq[\w]^\w$, there exists an almost disjoint family $\{A_\alpha\}_{\alpha\in\kappa}\subseteq[\w]^\w$ such that $A_\alpha\subseteq I_\alpha$ for every $\alpha\in\kappa$.
\end{lemma}

\begin{lemma}\label{t:thinr} If $\mathfrak r<\mathfrak c$, then there exists an almost Boolean subgroup $G\subseteq S_\w$ of cardinality $|G|=\mathfrak r$ such that the finitary coarse space $(\w,\E_G)$ is cellular and $2$-thin but $\w$ cannot be covered by finitely many $1$-thin subspaces.
\end{lemma}

\begin{proof} By definition of the cardinal $\mathfrak r$, there exists a family $\{R_\alpha\}_{\alpha\in\mathfrak r}\subseteq[\w]^\w$ such that for any finite partition of $\w$ one of the cells of the partition contains the set $R_\alpha$ for some $\alpha\in\mathfrak r$. 

By Lemma~\ref{l:a'}, there exists an almost disjoint family of infinite sets $(A_\alpha)_{\alpha\in\mathfrak r}$ in $\w$ such that $A_\alpha\subseteq R_\alpha$ for every $\alpha\in\mathfrak r$. For every $\alpha\in\mathfrak r$ choose a bijective function $f_\alpha\in S_\w$ such that $f_\alpha\circ f_\alpha$ is the indentity map of $\w$ whose support $\supp(f_\alpha)=\{x\in\w:f_\alpha(x)\ne x\}$ coincides with the infinite set $A_\alpha$.

Let $G$ be the subgroup of $S_\w$, generated by the set of involutions $\{f_\alpha\}_{\alpha\in\mathfrak r}$. Taking into account that the family $(\supp(f_\alpha))_{\alpha\in\mathfrak r}$ is almost disjoint, we can show that the group $G$ is almost Boolean and the coarse space $(G,\E_G)$ is cellular and $2$-thin.

Next, we show that $\w$ cannot be covered by finitely many $1$-thin sets. Assuming that such a finite cover exists, we can find $\alpha\in\mathfrak r$ such that the set $R_\alpha$ is contained in one of the cells of the cover and hence $R_\alpha$ is 1-thin. On the other hand, for the entourage $E_\alpha=\Delta_\w\cup\{(x,f_\alpha(x)):x\in \w\}$ the set $\{x\in\w:R_\alpha\cap E_\alpha(x)\ne \{x\}\}$ contains the infinite set $A_\alpha$, which means that the set $R_\alpha$ is not $1$-thin. 
\end{proof}
\end{proof}

\begin{problem} Is there a ZFC-example of a subgroup $G\subseteq S_\w$ of cardinality $|G|=\w_1$ such that the finitary coarse space $(\w,\E_G)$ is $2$-thin but cannot be covered by finitely many $1$-thin subspaces.
\end{problem}

\begin{theorem}\label{t:thin-n} For every number $n\in\mathbb N$ there exists an almost Boolean subgroup $G\subseteq S_\w$ of cardinality $|G|=\mathfrak u$ such that
\begin{enumerate}
\item the finitary coarse space $(\w,\E_G)$ is cellular and $2$-thin;
\item $(\w,\E_G)$ can be covered by $n$ thin sets;
\item $(\w,\E_G)$ cannot be covered by $(n-1)$ thin sets.
\end{enumerate}
\end{theorem}

\begin{proof}  Write the set $\w$ as the union $\w=\bigcup_{i\in n}\Omega_i$ where $\Omega_0=\{nk:k\in\w\}$ and $\Omega_i=\Omega_0+i$ for $i\in n$. 

By Proposition 8.1 in \cite{Blass}, there exists an almost disjoint family $\A\subseteq[\Omega_0]^\w$ of cardinality $|\A|=\mathfrak c$. Consider the   filter $$\mathcal F=\{X\in[\Omega_0]^\w:|\{A\in\A:A\not\subseteq^* X\}|<\mathfrak c\}.$$

Fix any free ultrafilter $\U$ on $\Omega_0$ possessing a base $\{U_\alpha\}_{\alpha\in\mathfrak u}$ of cardinality $\mathfrak u$. If $\mathfrak u=\mathfrak c$, then we can additionally assume that $\F\subseteq\U$.

\begin{lemma} There exists an almost disjoint family $\{A_\alpha\}_{\alpha\in\mathfrak u}\subseteq[\Omega_0]^\w$ such that $A_\alpha\subseteq U_\alpha$ for all $\alpha\in\mathfrak u$.
\end{lemma}

\begin{proof} If $\mathfrak u<\mathfrak c$, then the existence of the family $(A_\alpha)_{\alpha\in\mathfrak u}$ follows from Lemma~\ref{l:a'}. So, we assume that $\mathfrak u=\mathfrak c$. In this case the ultrafilter $\U$ contains the filter $\F$.

By transfinite induction, for every $\alpha\in\mathfrak c$ choose a set $A'_\alpha\in\A\setminus\{A'_\beta\}_{\beta<\alpha}$ such that the set $A_\alpha=A_\alpha'\cap U_\alpha$ is infinite. The choice of the set $A'_\alpha$ is always possible since the family $\{A\in\A:|A\cap U_\alpha|=\w\}$ has cardinality of continuum (in the opposite case, the set $\w\setminus U_\alpha$  belongs to the filter $\F\subseteq\U$). 
\end{proof}

Consider the finite family $$[n]^2=\{(i,j)\in n\times n:i<j\},$$ which can be identified with the family of 2-element subsets of the ordinal $n=\{0,\dots,n-1\}$.
 
For every $\alpha\in\mathfrak u$ and $p\in[n]^2$ choose an infinite set $A_{\alpha,p}\subseteq A_\alpha\subseteq\Omega_0$ such that the family $(A_{\alpha,p})_{p\in[n]^2}$ is disjoint. Consider the involution $f_\alpha:\w\to\w$ defined by the formula
$$f_\alpha(x)=\begin{cases} x+(j-i)&\mbox{if $x-i\in A_{\alpha,p}$ for some $p=(i,j)\in[n]^2$};\\
x-(j-i)&\mbox{if $x-j\in A_{\alpha,p}$ for some $p=(i,j)\in[n]^2$};\\
x&\mbox{otherwise}.
\end{cases}
$$

Let $G\subseteq S_\w$ be the subgroup generated by the set of involutions $\{f_\alpha\}_{\alpha\in\mathfrak u}$. 

Taking into account that the family $(\supp(f_\alpha))_{\alpha\in\mathfrak u}$ is almost disjoint, one can show that the group $G$ is almost Boolean and the coarse space $(\w,\E_G)$ is cellular and $2$-thin. The choice of the involutions $f_\alpha$ guarantees that for every $i\in n$ the set $\Omega_i$ is thin. Therefore, $\w=\bigcup_{i\in n}\Omega_i$ is the union of $n$ thin sets $\Omega_0,\dots,\Omega_{n-1}$.

Now take any partition $\{P_1,\dots,P_{n-1}\}$ of $\w$ into $n-1$ pieces. For every $i\in n$ consider the ultrafilter $\U_i=\{V\subseteq \w:\exists U\in\U\;\;U+i\subseteq V\}$ on $\w$, and observe that for some number $k_i\in\{1,\dots,n\}$ the set $P_{k_i}$ belongs to $\U_i$. By the Pigeonhole principle, there exists an index $k\in\{1,\dots,n-1\}$ such that the set $\{i\in n:k_i=k\}$ contains two numbers $i<j$. Hence $P_k\in \U_i\cap \U_j$ and there exists $\alpha\in \mathfrak u$ such that $U_\alpha+\{i,j\}\subseteq P_k$. Consider the pair $p=(i,j)\in[n]^2$ and observe that for the entourage $E_\alpha=\Delta_\w\cup\{(x,f_\alpha(x)):x\in \w\}$ the set $\{x\in P_k:P_k\cap E_\alpha(x)\ne\{x\}\}$ contains the set $A_{\alpha,p}+\{i,j\}$ and hence is infinite, witnessing that $P_k$ is not thin.
\end{proof}

\begin{problem} Let $G$ be a group of cardinality $\aleph_1$ endowed with the finitary coarse group structure $\E_G$. Can every $n$-thin subset of G be partitioned into $n$ thin subsets? {\rm This is so if $G$ is Abelian, see \cite{PS15}.}
\end{problem}



\section{ Coarse groups }\label{s:groups}

A topological space $X$ with no isolated points is called {\it maximal} if $X$ has an isolated point in every stronger topology. 
Under CH, there exists a  maximal topological group \cite{b8}.  
On the other hand, if there exists a maximal  topological group, then there is a $P$-point in $\beta\w\setminus\w$  \cite{b9}.

If  a topological group $G$ contains an infinite totally bounded subset, then there exists a non-closed discrete subset of $G$ \cite{b10}.  
Every discrete subset of a maximal group is closed. 
Answering a question from \cite{b10},  Reznichenko and Sipacheva \cite{b11} proved that if there exists a countable non-discrete topological group in which every discrete subset is closed, then there is a rapid ultrafilter  on $\omega$.

For a group $G$, a family $\mathcal{I}$ of subset of $G$ is called \cite{b12} a {\em group ideal} if  $\bigcup\I=G$,  $ \mathcal{I}$   is closed under taking subsets and finite unions, and $A, B\in \mathcal{I}$  implies $AB^{-1}\in  \mathcal{I}$. 
Every group ideal $ \mathcal{I}$ defines a coarse structure  on $G$ with the base  $\big\{\{(x, Fx): x\in G\} :1_G\in F\in \mathcal{I} \big\}$  and $G$  endowed with this coarse structure is called a {\it  coarse group}. 

   A coarse group $(G, \mathcal{E})$ is called {\it maximal }   $(G, \mathcal{E})$ is unbounded but $G$  is bounded in every (not necessary group)  coarse structure on $G$, which is strictly larger that $\E$. 

Under CH, there exists a maximal coarse  Boolean group \cite{b13}.

\begin{problem}[\cite{b13}] Does there exist a maximal coarse group in ZFC?
\end{problem}

Every maximal coarse group $(G, \mathcal{E})$ is large \cite{b13} and hence indiscrete. 

\begin{problem} In ZFC, does there exist an indiscrete coarse group $(G,\E)$?
\end{problem}

Let $G$ be a totally bounded topological group endowed with the finitary coarse group structure $\E_G$. Then there exist $\E_G$-discrete subsets $A, B\subseteq G$ such that $A$ is dense and $B$ has the unique    limit point \cite{b14}.
For open problems concerning $\E_G$-discrete subsets of topological groups, see \cite{b14}, where the following problem is posed.

\begin{problem}  In ZFC, does there exist a countable non-discrete topological group $G$ in which every $\E_G$-discrete subset is closed?
\end{problem}

\section{Acknowledgement}

The authors would like to express their thanks to Santi Spadaro for the information on the result \cite{BHM} of Baumgartner, Hajnal and M\'at\'e, which was essentially used in the proofs of Theorems~\ref{t:thin} and \ref{t:thin-n}. 

\end{document}